\documentclass[11pt]{amsart}
\usepackage{amssymb}
\usepackage{calrsfs}
\usepackage[all]{xy}
\usepackage{tikz}
\usepackage{xcolor}
\usepackage[letterpaper,margin=1in]{geometry}

\usepackage{hyperref}

\colorlet{revision}{black}

\newtheorem{dummy}{dummy}[section]

\newtheorem{theorem}[dummy]{Theorem}

\newtheorem{corollary}[dummy]{Corollary}
\newtheorem{proposition}[dummy]{Proposition}
\theoremstyle{definition}
\newtheorem{definition}[dummy]{Definition}

\newtheorem{remark}[dummy]{Remark}

\DeclareMathOperator{\Fun}{Fun}

\newcommand{\id}{\mathrm{id}}

\newcommand{\LH}{L^\mathrm{H}}
\newcommand{\Relcat}{\mathrm{RelCat}}
\newcommand{\sspc}{\mathrm{sSpc}}
\newcommand{\sset}{\mathrm{sSet}}

\newcommand{\ainf}{\mathrm{A_\infty}}
\newcommand{\dgcat}{\mathrm{DGCat}}
\newcommand{\ainfcat}{\mathrm{A_\infty Cat}}
\newcommand{\ainfalg}{\mathrm{A_\infty Alg}}
\newcommand{\wainf}{W^\ainf_\mathrm{qe}}
\newcommand{\wdg}{W^\mathrm{DG}_\mathrm{qe}}
\newcommand{\fib}{\mathrm{fib}}

\title[DG versus A-infinity]{Remarks on the equivalence between differential graded categories and A-infinity categories}

\author{James Pascaleff}
\address{James Pascaleff, University of Illinois at Urbana-Champaign}
\email{jpascale@illinois.edu}

\begin{document}
\maketitle

\begin{abstract}
  We show that the homotopy theories of differential graded categories and $\ainf$-categories over a field are equivalent at the $(\infty,1)$-categorical level. The results are corollaries of a theorem of Canonaco-Ornaghi-Stellari combined with general relationships between different versions of $(\infty,1)$-categories.
\end{abstract}

\section{Introduction}

The purpose of this note is to show that certain desirable propositions follow from a direct combination of published results. The propositions state in various forms a ``homotopy equivalence'' between the ``homotopy theory of differential graded (DG) categories'' and the ``homotopy theory of $\ainf$-categories.''

The equivalence between DG categories and $\ainf$-categories has been studied for several decades, starting no later than the thesis of Lef\`{e}vre-Hasegawa \cite{lefevre}. Without recounting the whole history, we shall begin our discussion with one of the more recent versions of this equivalence, proven by Canonaco-Ornaghi-Stellari \cite{COS}.

We fix a field $k$ over which our categories are linear. Let $\dgcat$ denote the category whose objects are small DG categories and whose morphisms are DG functors. Let $\ainfcat$ denote the category whose objects are small strictly unital $\ainf$-categories and whose morphisms are strictly unital $\ainf$-functors. Both $\dgcat$ and $\ainfcat$ are ordinary (rather than homotopy coherent) categories. They are not small categories, but all of the categorical constructions we will perform on them may be construed as happening within the framework of Grothendieck universes.

A DG category is, by definition, the same as an $\ainf$-category with vanishing higher operations. Likewise, a DG functor between DG categories is the same as an $\ainf$-functor with vanishing higher components. Thus there is a functor
\begin{equation}
  i : \dgcat \to \ainfcat.
\end{equation}

To a DG or $\ainf$-category $A$ we may associate a graded $k$-linear category $H(A)$ by replacing all hom-complexes by their cohomologies; if we retain only cohomology in degree zero, we obtain a $k$-linear category $H^{0}(A)$. A DG or $\ainf$-functor $F : A \to B$ induces functors $H(F) : H(A) \to H(B)$ and $H^{0}(F) : H^{0}(A) \to H^{0}(B)$.

A DG or $\ainf$-functor $F: A \to B$ is a \emph{quasi-equivalence} if $H(F)$ and $H^{0}(F)$ are both equivalences of categories; this is equivalent to the conjunction of the conditions that $F$ acts by quasi-isomorphisms on all hom-complexes and that $H^{0}(F)$ is essentially surjective. We denote by $\wdg$ the subcategory of $\dgcat$ with the same objects but with only the DG quasi-equivalences as morphisms; similarly $\wainf$ denotes the subcategory of $\ainfcat$ with the same objects but with only the $\ainf$-quasi-equivalences as morphisms.

We now state the results of Canonaco-Ornaghi-Stellari \cite{COS}. First is the construction of a functor
\begin{equation}
  U : \ainfcat \to \dgcat
\end{equation}
which could be called a \emph{rectification functor}. On objects, this functor takes an $\ainf$-category $A$, forms the bar construction $BA$, which is a ``DG cocategory'', and then forms the cobar construction $\Omega(BA)$ of that. The precise details are not necessary for our discussion, but only the following theorems.

\begin{theorem}[Canonaco-Ornaghi-Stellari \cite{COS}]
  \label{thm:cos-main}
  The functors $U : \ainfcat \rightleftarrows\dgcat : i$ form an adjunction such that
  \begin{enumerate}
  \item $U$ is left adjoint to $i$,
  \item $i$ takes $\wdg$ into $\wainf$,
  \item $U$ takes $\wainf$ into $\wdg$,
  \item all components of the unit $\eta : \id_{\ainfcat} \to iU$ belong to $\wainf$, and
  \item all components of the counit $\epsilon : Ui \to \id_{\dgcat}$ belong to $\wdg$.
  \end{enumerate}
\end{theorem}
\begin{proof}
  All of these statements are contained in \cite[Proposition 2.1]{COS}, except for the statement that $i$ takes $\wdg$ into $\wainf$, which is obvious, and the statement that $U$ takes $\wainf$ into $\wdg$, which is demonstrated in the proof of \cite[Theorem 2.2]{COS}.
\end{proof}

Consider now the localizations $\dgcat[(\wdg)^{-1}]$ and $\ainfcat[(\wainf)^{-1}]$ formed by inverting the quasi-equivalences---these are the \emph{homotopy categories} of DG categories and $\ainf$-categories respectively.
\begin{theorem}[\cite{COS}, Theorem 2.2]
  \label{thm:1cat-equiv}
  The functors $U$ and $i$ induce mutually quasi-inverse equivalences of homotopy categories
  \begin{equation}
    \ainfcat[(\wainf)^{-1}] \rightleftarrows \dgcat[(\wdg)^{-1}].
  \end{equation}
\end{theorem}

Theorem \ref{thm:1cat-equiv} follows directly from the formal properties stated in Theorem \ref{thm:cos-main}. Theorem \ref{thm:1cat-equiv} states that ``the homotopy theory of DG categories'' is equivalent to ``the homotopy theory of $\ainf$-categories'' at the 1-categorical level: the homotopy categories are equivalent as ordinary categories.

\textbf{One may then ask whether ``the homotopy theory of DG categories'' is equivalent to ``the homotopy theory of $\ainf$-categories'' at the $(\infty,1)$-categorical level. We shall show both that the answer is affirmative, and that this already follows from Theorem \ref{thm:cos-main} without any further work involving DG or $\ainf$-categories.} That is to say, Theorem \ref{thm:cos-main} implies not only Theorem \ref{thm:1cat-equiv}, but also an equivalence of homotopy theories at the apparently stronger $(\infty,1)$-categorical level. We give several precise formulations of this result. They are Corollaries \ref{cor:dk-equiv}, \ref{cor:bk-equiv}, \ref{cor:css-equiv}, \ref{cor:dg-vs-ainf-quasicat}, \ref{cor:coherent-nerve-equiv}, and \ref{cor:qcat-localization-equiv}.

\subsection{Context}

The results of this note are presumably ``folklore'' or ``well-known to experts.'' On the other hand, I was not able to find any reference for these results, so I went ahead and put the pieces of the puzzle together myself. The main content of this note is the observation that a certain kind of adjunction between relative categories (called a ``Dwyer-Kan adjunction of relative categories'' below) implies an equivalence between their localizations at the $(\infty,1)$-categorical level.

I have not attempted to survey all of the literature that touches on the questions raised here, of which there is a great deal, both about DG and $\ainf$-categories, and about relationships between different versions of $(\infty,1)$-categories. Nevertheless, I should note that the article of Vallette \cite{vallette} was very helpful. To relate different versions of $(\infty,1)$-categories, I have attempted to cite the original sources for the various functors that we use, but the overall picture that emerges is essentially the same as the one presented by Hinich \cite{hinich}.

One motivation for this note was a 2020 preprint of Oh-Tanaka.\footnote{Yong-Geun Oh and Hiro Lee Tanaka. A-infinity-categories, their infinity-category, and their localizations. arXiv:2003.05806.} I had wished to cite Theorem 1.1 from that preprint, but unfortunately it was withdrawn. Note that Oh and Tanaka work over a general commutative base ring. The results of this note show that---\emph{in the special case where the base ring is a field}---Theorem 1.1 from that preprint is in fact true. However, the precise form of the equivalence appears to be different from the one proposed by Oh-Tanaka. Tanaka has also informed me of ongoing work to recover Theorem 1.1 over a general commutative base ring.

\subsection{Acknowledgements}
I wish to thank Charles Rezk and Doron Grossman-Naples for very helpful conversations, and Hiro Lee Tanaka for his comments on this note. \textcolor{revision}{I also thank the referee for their comments.}

\section{Dywer-Kan equivalence}

In their works \cite{dk-localization, dk-calculating}, Dwyer and Kan studied a simplicially enriched version of the localization of category at a subcategory. We shall isolate the concepts and results that we need.

By definition, a \emph{relative category} is a pair $(C,W)$ consisting of a category $C$ and a subcategory $W$, subject only to the condition that $W$ contain all objects (and hence all identity morphisms) of $C$. We refer to the morphisms in $W$ as \emph{weak equivalences}.

A \emph{relative functor} $F : (C_{1},W_{1}) \to (C_{2},W_{2})$ is an ordinary functor $F: C_{1} \to C_{2}$ such that $F(W_{1}) \subseteq W_{2}$.

Recall that an adjunction $L \dashv R$ consists of two functors $L : C \to D$, $R : D \to C$ and natural transformations $\eta : \id_{C} \to RL$ and $\epsilon : LR \to \id_{D}$ called the unit and counit respectively, satisfying certain relations known as triangle identities.

An adjoint equivalence is an adjunction such that both the unit and counit are natural isomorphisms. This condition obviously implies that $L$ and $R$ are equivalences of categories.

We now isolate an important concept from \cite{dk-calculating}; I am not aware of a standard term for this concept.

\begin{definition}
  \label{defn:dk-adjunction}
  Let $(C_{1},W_{1})$ and $(C_{2},W_{2})$ be relative categories, and let $(L,R,\eta,\epsilon)$ be an adjunction, where $L : C_{1} \to C_{2}$ and $R : C_{2} \to C_{1}$. We say that $(L,R,\eta,\epsilon)$ is a \emph{Dwyer-Kan adjunction of relative categories} if the following conditions hold:
  \begin{enumerate}
  \item $L(W_{1}) \subseteq W_{2}$ and $R(W_{2})\subseteq W_{1}$, so that $L$ and $R$ are relative functors,
  \item for every object $X$ in $C_{1}$, the component of the unit $\eta_{X} : X \to RLX$ is in $W_{1}$, and
  \item for every object $Y$ in $C_{2}$, the component of the counit $\epsilon_{Y} : LRY \to Y$ is in $W_{2}$.
  \end{enumerate}
\end{definition}

In the special case where $W_{i} = C_{i}$ for $i = 1,2$, a Dwyer-Kan adjunction of relative categories is just an adjunction. 
In the special case where $W_{i}$ consists precisely of the isomorphisms in $C_{i}$, a Dwyer-Kan adjunction of relative categories is the same as an adjoint equivalence.  Alternately, a Dwyer-Kan adjunction of relative categories is ``the same as'' an adjoint equivalence where the phrase ``natural isomorphism'' is replaced by ``natural transformation in $W$.''

\begin{theorem}[Restatement of \ref{thm:cos-main}]
  The functors $U : (\ainfcat,\wainf) \rightleftarrows(\dgcat,\wdg) : i$ form a Dwyer-Kan adjunction of relative categories.
\end{theorem}

Next, given a relative category $(C,W)$, we may form the localization $C[W^{-1}]$ as an ordinary category. In 1980, Dwyer and Kan \cite{dk-localization, dk-calculating} gave two constructions of simplicially enriched categories that enhance $C[W^{-1}]$ to what contemporary mathematicians consider to be a kind of $(\infty,1)$-category. In this sense, a relative category is perhaps the most elementary data structure that presents an $(\infty,1)$-category.

The version that we shall use is the so-called \emph{hammock localization}. Given a relative category $(C,W)$, this construction produces a category enriched over simplicial sets $\LH(C,W)$. The objects of $\LH(C,W)$ are the same as those of $C$, and for each pair of objects $X,Y$ there is a morphism simplicial set $\LH(C,W)(X,Y)$ whose simplicies correspond to certain diagrams called hammocks; the detailed definition will not concern us.

The simpicial sets $\LH(C,W)(X,Y)$ are to be regarded as spaces and as such they have homotopy groups. The category $\pi_{0}\LH(C,W)$ is obtained by replacing each morphism space by its set of components. According to \cite[Proposition 3.1]{dk-calculating}, there is a canonical isomorphism of categories
\begin{equation}
  \pi_{0}\LH(C,W) \cong C[W^{-1}].
\end{equation}

An important property of hammock localization is that it is manifestly functorial with respect to relative functors. That is, if $F : (C_{1},W_{1}) \to (C_{2},W_{2})$ is a relative functor, then there is an induced simplicially enriched functor $\LH F : \LH(C_{1},W_{1}) \to \LH(C_{2},W_{2})$.

We recall the appropriate notion of ``weak homotopy equivalence'' for simplicially enriched categories. Recall that a \emph{weak homotopy equivalence of simplicial sets} is a map that, after passing to geometric realizations, induces a bijection on $\pi_{0}$ and an isomorphism on all homotopy groups with all possible basepoints.
\begin{definition}
  Let $D_{1}$ and $D_{2}$ be simplicially enriched categories, and let $F : D_{1} \to D_{2}$ be a simplicially enriched functor. Then $F$ is a \emph{Dwyer-Kan equivalence} if the following conditions hold:
  \begin{enumerate}
  \item The induced functor $\pi_{0}F : \pi_{0}D_{1} \to \pi_{0}D_{2}$ is an equivalence of categories, and
  \item for every pair of objects $X,Y$ in $D_{1}$, the map of simplicial sets $F : D_{1}(X,Y) \to D_{2}(FX,FY)$ is a weak homotopy equivalence of simplicial sets. 
  \end{enumerate}
\end{definition}

The following proposition asserts a relationship between Dwyer-Kan adjunctions of relative categories and Dwyer-Kan equivalences of localizations.

\begin{proposition}
  \label{prop:dk-adjunction-implies-dk-equivalence}
  Let $(C_{1},W_{1})$ and $(C_{2},W_{2})$ be relative categories, and let $(L,R,\eta,\epsilon)$ be a Dwyer-Kan adjunction of relative categories, where $L : C_{1} \to C_{2}$ and $R: C_{2} \to C_{1}$. Then
  \begin{equation*}
    \LH(L) : \LH(C_{1},W_{1}) \to \LH(C_{2},W_{2}) \quad \text{and} \quad \LH(R) :  \LH(C_{2},W_{2}) \to \LH(C_{1},W_{1})
  \end{equation*}
  are Dwyer-Kan equivalences of simplicially enriched categories. Also,
  \begin{equation*}
    \pi_{0}\LH(L) : C_{1}[W_{1}^{-1}] \to C_{2}[W_{2}^{-1}]\quad \text{and} \quad \pi_{0}\LH(R) : C_{2}[W_{2}^{-1}] \to C_{1}[W_{1}^{-1}]
  \end{equation*}
  are mutually quasi-inverse equivalences of categories.
\end{proposition}

\begin{proof}
  This is a restatement in our terminology of \cite[Corollary 3.6]{dk-calculating}.
\end{proof}

We remark that $\pi_{0}\LH(L)$ and $\pi_{0}\LH(R)$ actually fit into an adjoint equivalence of the ordinary localizations, with unit and counit induced from $\eta$ and $\epsilon$.

According to Bergner \cite{bergner-model-structure}, there is a model category structure on the category of simplicially enriched categories such that the weak equivalences are precisely the Dwyer-Kan equivalences. So Proposition \ref{prop:dk-adjunction-implies-dk-equivalence} states that $\LH(C_{1},W_{1})$ and $\LH(C_{2},W_{2})$ are weakly equivalent in the Bergner model structure.

\begin{corollary}
  \label{cor:dk-equiv}
  The functors
  \begin{equation}
    \LH(U) : \LH(\ainfcat,\wainf) \rightleftarrows \LH(\dgcat,\wdg) : \LH(i)
  \end{equation}
  are both Dwyer-Kan equivalences of simplicially enriched categories. That is to say, $\LH(U)$ and $\LH(i)$ are both weak equivalences for the Bergner model structure on the category of simplicially enriched categories.
\end{corollary}

\section{The Barwick-Kan model structure on $\Relcat$}

Let $\Relcat$ denote the category whose objects are small relative categories and whose morphisms are relative functors. Instead of taking simplicial localization of relative categories as above, we may regard relative categories themselves as a kind of $(\infty,1)$-categories. In this vein Barwick-Kan \cite{barwick-kan-equiv,barwick-kan-model} defined a model category structure directly on $\Relcat$. This model category structure is designed so as to have a specific relationship to Rezk's theory of complete Segal spaces \cite{rezk}; we denote by $\sspc$ the category of simplicial spaces (that is, bisimplicial sets) on which Rezk's complete Segal space model structure is defined.

\begin{theorem}[\cite{barwick-kan-equiv,barwick-kan-model}]
  \label{thm:bk-main}
  There is a model structure on $\Relcat$, called the \emph{Barwick-Kan model structure}, with the following properties:
  \begin{enumerate}
  \item there is an adjunction $K_{\xi} : \sspc \rightleftarrows \Relcat: N_{\xi}$ where $K_{\xi}$ is left adjoint to $N_{\xi}$, and this adjunction is a Quillen equivalence between the Barwick-Kan model structure on $\Relcat$ and Rezk's complete Segal space model structure on $\sspc$,
  \item a morphism in $\Relcat$ is a weak equivalence if and only if its image under $N_{\xi}$ is a weak equivalence in $\sspc$, and
  \item a morphism $F : (C_{1},W_{1}) \to (C_{2},W_{2})$ is a weak equivalence if and only if $\LH F : \LH(C_{1},W_{1})\to \LH(C_{2},W_{2})$ is a Dwyer-Kan equivalence of simplicially enriched categories.
  \end{enumerate}
\end{theorem}

\begin{proof}
  This is a simplified paraphrase of \cite[Theorem 6.1]{barwick-kan-model} and \cite[Theorem 1.8]{barwick-kan-equiv}.
\end{proof}

\begin{corollary}
  \label{cor:bk-equiv}
  The relative functors $U : (\ainfcat,\wainf) \rightleftarrows (\dgcat,\wdg) : i$ are both weak equivalences for the Barwick-Kan model structure on $\Relcat$. Also, the functors $N_{\xi}(U)$ and $N_{\xi}(i)$ are weak equivalences for Rezk's complete Segal space model structure on $\sspc$.
\end{corollary}

\section{Fibrant replacement, complete Segal spaces, and quasicategories}

So far we have considered relative categories and simplicially enriched categories as models for $(\infty,1)$-categories. It is now understood that many different models of $(\infty,1)$-categories are equivalent: see \cite{bergner-survey,bergner-book} for surveys of these relationships.

Two further models for $(\infty,1)$-categories are complete Segal spaces and quasicategories. The former has already made a passing appearance. The phrase ``Rezk's complete Segal space model structure'' refers to a model structure on $\sspc$, but a \emph{complete Segal space} is a fibrant object for this model structure. Similarly, there is a model structure on the category $\sset$ of simplicial sets called the \emph{Joyal model structure},  and a \emph{quasicategory} is a fibrant object for the Joyal model structure.

This means that in order to convert $(\dgcat,\wdg)$ and $(\ainfcat,\wainf)$ into complete Segal spaces or quasicategories, we must consider fibrant replacements. \textcolor{revision}{In general, if $\mathcal{M}$ is a model category and $X$ is an object of $\mathcal{M}$, we shall denote by $X^{\fib}$ a fibrant replacement of $X$, which is a fibrant object $X^{\fib}$ of $\mathcal{M}$ equipped with an acyclic cofibration $j_{X} : X \to X^{\fib}$. Given a morphism $\phi: X\to Y$ in $\mathcal{M}$, and fibrant replacements $j_{X}:X \to X^{\fib}$ and $j_{Y} : Y \to Y^{\fib}$, the composite $j_{Y}\phi$ is a morphism whose target object is fibrant. Therefore it extends along the acyclic cofibration $j_{X}$ to a morphism $\phi^{\fib} : X^{\fib} \to Y^{\fib}$ such that $\phi^{\fib}j_{X} = j_{Y}\phi$. The morphism $\phi^{\fib}$ is uniquely determined up to homotopy. If $\phi$ is a weak equivalence, then $\phi^{\fib}$ is as well, by the two-out-of-three property.}

A natural question is whether the objects of $\Relcat$ we are interested in, namely $(\dgcat,\wdg)$ and $(\ainfcat,\wainf)$, are already fibrant for the Barwick-Kan model structure on $\Relcat$. We have the following proposition.

\begin{proposition}
  The object $(\dgcat,\wdg)$ is a fibrant object in $\Relcat$.
\end{proposition}
\begin{proof}
  According to Tabuada \cite{tabuada}, $\dgcat$ itself admits the structure of a model category such that $\wdg$ is the class of weak equivalences. The desired conclusion then follows from a result of Meier \cite[Main Theorem]{meier}.
\end{proof}

\begin{remark}
  \label{rem:fibrancy}
  I was unable to find any result implying that $(\ainfcat,\wainf)$ is fibrant in $\Relcat$, but this may very well be true. See Section \ref{sec:fibrancy} for further discussion.
\end{remark}

\textcolor{revision}{In order to proceed, let us choose a fibrant replacement $j_{\ainf}: (\ainfcat,\wainf) \to (\ainfcat,\wainf)^{\fib}$ in $\Relcat$. As explained above, we obtain morphisms
  \begin{equation}
    U^{\fib} : (\ainfcat,\wainf)^{\fib} \rightleftarrows (\dgcat,\wdg) : i^{\fib}
  \end{equation}
  such that $i^{\fib} = j_{\ainf}i$ and $U^{\fib}j_{\ainf} = U$. Both $i^{\fib}$ and $U^{\fib}$ are weak equivalences between fibrant objects in $\Relcat$. }

Recall that a right Quillen functor (= right adjoint in a Quillen adjunction) always preserves fibrant objects and weak equivalences between fibrant objects.

\begin{corollary}
  \label{cor:css-equiv}
  Let $N_{\xi} : \Relcat \to \sspc$ be the right Quillen functor appearing in Theorem \ref{thm:bk-main}. Then $N_{\xi}(\dgcat,\wdg)$ and $N_{\xi}(\ainfcat,\wainf)^{\fib}$ are complete Segal spaces, \textcolor{revision}{and $N_{\xi}(i^{\fib})$ and $N_{\xi}(U^{\fib})$ are equivalences} between them.
\end{corollary}

To pass to quasicategories we use the following result of Joyal-Tierney \cite{joyal-tierney}.
\begin{theorem}[\cite{joyal-tierney}]
  \label{thm:jt-main}
  There is a pair of adjoint functors $p_{1}^{*} : \sset \rightleftarrows \sspc : i_{1}^{*}$, with $p_{1}^{*}$ left adjoint to $i_{1}^{*}$, which is a Quillen equivalence between the Joyal model structure on $\sset$ and Rezk's complete Segal space model structure on $\sspc$.
\end{theorem}

By composing the Quillen equivalence in Theorem \ref{thm:jt-main} with the one in Theorem \ref{thm:bk-main}, we obtain a Quillen equivalence between $\Relcat$ with the Barwick-Kan model structure and $\sset$ with the Joyal model structure (such that $\Relcat$ is the domain of the right adjoint). This equivalence was also explicitly described by Barwick-Kan \cite{barwick-kan-thomason}.

\begin{corollary}
  \label{cor:dg-vs-ainf-quasicat}
  The simpicial sets $i_{1}^{*}N_{\xi}(\dgcat,\wdg)$ and $i_{1}^{*}N_{\xi}(\ainfcat,\wainf)^{\fib}$ are quasicategories, \textcolor{revision}{and $i_{1}^{*}N_{\xi}(i^{\fib})$ and $i_{1}^{*}N_{\xi}(U^{\fib})$ are equivalences} between them.
\end{corollary}

Let us now summarize what we have achieved. Corollary \ref{cor:dg-vs-ainf-quasicat} is a lift of Theorem \ref{thm:1cat-equiv} to the level of quasicategories. Indeed, all of the functors that we are applying preserve the homotopy category $C[W^{-1}]$ of a relative category $(C,W)$. The homotopy category of the quasicategory $i_{1}^{*}N_{\xi}(\dgcat,\wdg)$ is $\dgcat[(\wdg)^{-1}]$, and likewise the homotopy category of $i_{1}^{*}N_{\xi}(\ainfcat,\wainf)^{\fib}$ is $\ainfcat[(\wainf)^{-1}]$. So the equivalence of categories in Theorem \ref{thm:1cat-equiv} follows from Corollary \ref{cor:dg-vs-ainf-quasicat} by taking the homotopy category.

\section{Other quasicategory models}

Since all models of $(\infty,1)$-categories are related, sometimes in several ways, there are other ``paths'' from $\Relcat$ to the category of quasicategories.

One path goes through the simplicial localization. Given a relative category $(C,W)$, we form $\LH(C,W)$. This may not be fibrant (the morphism simplicial sets may not be Kan complexes), so we should take a fibrant replacement $(\LH(C,W))^{\fib}$ for the Bergner model structure.\footnote{Two explicit constructions are given by applying to each morphism space either $\mathrm{Sing}| \cdot |$ or Kan's $\mathrm{Ex}^{\infty}$ functor.} Then we apply the homotopy coherent nerve $N_{c}$. Again this is a right Quillen functor in a Quillen equivalence between the Bergner model structure on simplicially enriched categories and the Joyal model structure on $\sset$, so $N_{c}(\LH(C,W))^{\fib}$ is a quasicategory.

\textcolor{revision}{We remark that the} quasicategory $N_{c}(\LH(C,W))^{\fib}$ is equivalent to $i_{1}^{*}N_{\xi}(C,W)^{\fib}$, the version considered previously \cite[p. 3277]{meier}, \cite{csp-mathoverflow}.

\textcolor{revision}{From Corollary \ref{cor:dk-equiv}, we have a pair of Dwyer-Kan equivalences $\LH(i)$ and $\LH(U)$ relating the simplicial localizations $\LH(\dgcat,\wdg)$ and $\LH(\ainfcat,\wainf)$. After fibrant replacement, we obtain
  \begin{equation}
    \LH(U)^{\fib} : \LH(\ainfcat,\wainf)^{\fib} \rightleftarrows \LH(\dgcat,\wdg)^{\fib} : \LH(i)^{\fib},
  \end{equation}
  where both morphisms are again weak equivalences for the Bergner model structure.}

\begin{corollary}
  \label{cor:coherent-nerve-equiv}
  \textcolor{revision}{The maps $N_{c}\LH(i)^{\fib}$ and $N_{c}\LH(U)^{\fib}$ are equivalences between the quasicategories $N_{c}(\LH(\dgcat,\wdg))^{\fib}$ and $N_{c}(\LH(\ainfcat,\wainf))^{\fib}$.}
\end{corollary}

It is also possible to take a relative category $(C,W)$ and directly produce a quasicategorical localization by localizing the nerve $N(C)$ with respect to the collection of edges $W$. Let us denote such a localization by $N(C)_{(W)}$. An explicit construction shown to me by Rezk consists of first taking the pushout of the diagram of simplicial sets
\begin{equation}
  \coprod_{w \in W} J \leftarrow \coprod_{w \in W} \Delta^{1} \rightarrow N(C),
\end{equation}
where $J$ is the walking isomorphism, and then taking fibrant replacement for the Joyal model structure. However, the object $N(C)_{(W)}$ is really intended to be characterized by a quasicategorical universal property, \textcolor{revision}{namely: for any quasicategory $D$, the map $\Fun(N(C)_{(W)},D) \to \Fun(N(C),D)$ induced by $N(C)\to N(C)_{(W)}$ is an equivalence of quasicategories from $\Fun(N(C)_{(W)},D)$ to the full sub-quasicategory of $\Fun(N(C),D)$ spanned by functors that take edges in $W$ to equivalences in $D$.}

\textcolor{revision}{Given a relative functor $F : (C_{1},W_{1}) \to (C_{2},W_{2})$, we obtain a map of nerves $N(F) : N(C_{1})\to N(C_{2}).$ Since $F(W_{1})\subseteq W_{2}$, this extends to a map of localizations $N(F)_{\mathrm{loc}} : N(C_{1})_{(W_{1})} \to N(C_{2})_{(W_{2})}$.}

\textcolor{revision}{The results of Hinich \cite{hinich, marino-mathoverflow} show that the previously considered versions $i_{1}^{*}N_{\xi}(C,W)^{\fib}$ and $N_{c}(\LH(C,W))^{\fib}$ are equivalent to $N(C)_{(W)}$. More precisely, given a relative category $(C,W)$, there is map of simplicial sets $N(C) \to N_{c}(\LH(C,W))^{\fib}$, which extends to a map $N(C)_{(W)} \to N_{c}(\LH(C,W))^{\fib}$, and this latter map is an equivalence of quasicategories. These equivalences are natural with respect to relative functors.}

\begin{corollary}
  \label{cor:qcat-localization-equiv}
  \textcolor{revision}{The maps $N(i)_{\mathrm{loc}}$ and $N(U)_{\mathrm{loc}}$ are equivalences between the quasicategorical localizations $N(\dgcat)_{(\wdg)}$ and $N(\ainfcat)_{(\wainf)}$}.
\end{corollary}


\section{The problem of fibrancy of $(\ainfcat,\wainf)$}
\label{sec:fibrancy}

We conclude with a few remarks about the problem raised in Remark \ref{rem:fibrancy}, namely the question of whether $(\ainfcat,\wainf)$ is a fibrant object in $\Relcat$. If this were true, it would slightly simplify some of the preceding constructions.

The main result of Meier \cite{meier} is that, in order for $(C,W)$ to be fibrant in $\Relcat$, it is sufficient that $(C,W)$ admits the structure of a \emph{fibration category}. This means that in addition to the class of weak equivalences $W$, one must also specify a subcategory $\mathrm{Fib} \subset C$ of \emph{fibrations}, such that a certain list of axioms is satisfied. See \cite[Definition 3.1]{meier} for the precise definition.

So in order to show that $(\ainfcat,\wainf)$ is fibrant, we would need a notion of fibration of $\ainf$-categories. A somewhat naive guess is the following: An $\ainf$-functor $F: A \to B$ between $\ainf$-categories is said to be a \emph{fibration} if the following conditions are satisfied:
\begin{enumerate}
\item for every pair of objects $X,Y$ in $A$ and each integer $p$, the first component $F^{1}: \hom_{A}^{p}(X,Y) \to \hom_{B}^{p}(FX,FY)$ is a surjective map of vector spaces, and
\item $H^{0}(F): H^{0}(A) \to H^{0}(B)$ is an isofibration.
\end{enumerate}
We recall that an \emph{isofibration} is a functor $F: C \to D$ between ordinary categories with the following property:
\begin{itemize}
\item If $\phi : F(X) \to Z$ is an isomorphism in $D$ whose source object is in the image of $F$, then there is an isomorphism $\tilde{\phi} : X \to \tilde{Z}$ in $C$ such that $F\tilde{\phi} = \phi$.
\end{itemize}
This definition of fibration of $\ainf$-categories is directly analogous both to the definition of fibration of simplicially enriched categories for the Bergner model structure, and to the definition of fibration of DG categories for the Tabuada model structure.\footnote{Even so, other definitions of ``fibration of $\ainf$-categories'' are possible. A less naive proposal is to take from \cite{lefevre} the fibrations of conilpotent DG coalgebras, generalize that concept to DG cocategories somehow, and then take the preimage of that class under the bar construction.}

The problem is now to verify the axioms (F1)--(F5) of \cite[Definition 3.1]{meier} using this notion of fibration. Axioms (F1), (F2), and (F3) are straightforward (note that $W$ and $\mathrm{Fib}$ both contain all isomorphisms and that all objects are fibrant). The factorization axiom (F5) follows from axiom (F4) and the existence of path objects in $\ainfcat$ \cite[Appendix D]{drinfeld-dg}.

The trickiest axiom turns out to (F4); in the present context, it is equivalent to the assertion that the pullback of a fibration of $\ainf$-categories along any $\ainf$-functor exists and is a fibration, and also that the pullback of an acyclic fibration (fibration which is also a weak equivalence) is an acyclic fibration. I do not know any proof of this assertion.

In fact the mere existence of such pullbacks is a nontrivial problem, because the category $\ainfcat$ is \textbf{not complete}. Indeed, the category $\ainfalg$ of $\ainf$-algebras ($\ainf$-categories with a single object) is not complete either, as was already well-appreciated by Lef\`{e}vre-Hasegawa \cite{lefevre}. See \cite[\S 1.5]{COS} for a simple counterexample. The reason is elementary: an $\ainf$-functor or $\ainf$-morphism of $\ainf$-algebras is precisely not ``a map of underlying vector spaces that preserves the operations''---it only preserves the operations up to specified homotopy. This means that the naive attempt to form the equalizer of a parallel pair of $\ainf$-morphisms $F,G : A \to B$---take the equalizer of $F^{1}$ and $G^{1}$ in vector spaces and then define $\ainf$-operations on the result---fails catastrophically. Instead, one must take the equalizer in the category of DG cocategories or DG coalgebras, and then prove that the resulting object lies in the essential image of the bar construction.

In the case of $\ainfalg$, one of the main results of \cite{lefevre} is that the category of conilpotent DG coalgebras admits a model structure such that the subcategory of fibrant objects is equivalent to $\ainfalg$. As explained in \cite{vallette}, it follows that $\ainfalg$ is a ``category of fibrant objects,'' implying that it is indeed a fibration category in the sense of \cite{meier}, where the weak equivalences are the $\ainf$-quasi-isomorphisms and the fibrations are degreewise surjections.

There is one more case that deserves mention. Let $O$ be set, and let $\ainfcat_{O}$ be the category whose objects are small $\ainf$-categories with object set $O$, and whose morphisms are $\ainf$-functors that are identity on $O$. As stated in \cite{COS}, the results of \cite{lefevre} apply \emph{mutatis mutandis} to this case. Thus the problem of extending the results of \cite{lefevre} to $\ainfcat$ is the simple fact that a functor, even if it is an equivalence, need not be a bijection on objects.


\bibliographystyle{alpha}
\bibliography{dg-versus-a-infinity}

\end{document}